\documentclass[12pt]{amsart}

\usepackage[colorlinks=true, pdfstartview=FitV, linkcolor=blue, citecolor=blue, urlcolor=blue, breaklinks=true]{hyperref}
\usepackage{amsmath,amsfonts,amssymb,amsthm,amscd,comment,paralist,stmaryrd,etoolbox,mathtools}
\usepackage[usenames,dvipsnames]{xcolor}
\usepackage{mdwlist}
\usepackage{braket}
\usepackage{slashed}
\usepackage{mdframed}

\newcommand\Z{\mathbb{Z}}
\newcommand\Q{\mathbb{Q}}

\newcommand\N{\mathbb{N}}

\newcommand\kk{\Bbbk}

\newcommand\dif{\mathrm{d}}

\newcommand\fh{\mathfrak{h}}

\newcommand\balpha{{\boldsymbol\alpha}}
\newcommand\bbeta{{\boldsymbol\beta}}
\newcommand\blambda{{\boldsymbol\lambda}}
\newcommand\bmu{{\boldsymbol\mu}}
\newcommand\bnu{{\boldsymbol\nu}}
\newcommand\bomega{{\boldsymbol\omega}}
\newcommand\bm{{\mathbf{m}}}

\newcommand\Sy{\mathrm{Sym}}

\newcommand\prt{\mathrm{part}}
\newcommand\clr{\mathrm{color}}

\newcommand\ts{}

\DeclareMathOperator{\qdim}{grdim}

\newtheorem{theo}{Theorem}[section]
\newtheorem{prop}[theo]{Proposition}
\newtheorem{lem}[theo]{Lemma}

\theoremstyle{definition}
\newtheorem{defin}[theo]{Definition}
\newtheorem{rem}[theo]{Remark}

\numberwithin{equation}{section}
\allowdisplaybreaks

\newtoggle{comments}    
\newtoggle{details}     
\newtoggle{prelimnote}  
\newtoggle{detailsnote} 

\toggletrue{detailsnote}

\iftoggle{comments}{%
  \newcommand{\comments}[1]{
    \begin{center}
      \parbox{6.5 in}{
        \color{blue}
          {\footnotesize \textbf{Comments:} #1}
        \color{black}}
    \end{center}}
}{%
  \newcommand{\comments}[1]{}
}

\iftoggle{details}{%
  \newcommand{\details}[1]{
      \ \\
      \color{OliveGreen}
        {\footnotesize \textbf{Details:} #1}
      \color{black}
      \\
  }
}{%
  \newcommand{\details}[1]{}
}

\iftoggle{prelimnote}{%
  \newcommand{\prelim}{\textsc{Preliminary version} \bigskip}
}{%
  \newcommand{\prelim}{}
}

\begin{document}
%

\title{Integral presentations of quantum lattice {H}eisenberg algebras}

\author{Diego Berdeja Su{\'a}rez}
\address{D.~Berdeja~Su{\'a}rez: Department of Mathematics and Statistics, University of Ottawa}
\email{diegobesu@hotmail.com}

\thanks{This research was supported by the MITACS Globalink program.}

\begin{abstract}
  We give integral presentations of quantum lattice Heisenberg algebras by viewing them as Heisenberg doubles.  Our presentations generalize those appearing previously in the literature.
\end{abstract}


\prelim

\maketitle
\thispagestyle{empty}
\tableofcontents

\section{Introduction}

The Heisenberg algebra has its origins in the description of the relationship between the momentum and position operators of quantum mechanics, as well as the creation and annihilation operators employed in the description of the quantum harmonic oscillator.  It also plays a vital role in Lie theory, especially in the field of vertex operators.

More recently, mathematicians have been developing so called \emph{categorifications} of the Heisenberg algebra and its quantum analogues (see \cite{Kho14,CL12,LS13,SV12,LS12,SY13,RS15a,RS15b}).  In categorification, generators with nice integrality properties are often an important ingredient.  The usual presentation of the Heisenberg algebra is in terms of a generating set that does not possess these properties.  Thus, the categorification literature often works in a different presentation.

The goal of the current paper is to develop explicit formulas for integral presentations of quantum lattice Heisenberg algebras (the Heisenberg algebras appearing most often in categorification).  While some special cases of these formulas have appeared before in the literature, the description in the current paper is more complete.  We obtain the integral presentations by adopting the point of the view of the Heisenberg double.  In particular, the Heisenberg algebra is the Heisenberg double of the Hopf algebra of symmetric functions.  Then our presentations arise from considering the complete and elementary symmetric functions as generators and performing computations in the Heisenberg double.  This approach towards quantum lattice Heisenberg algebras, followed, for example, by Rosso and Savage in \cite{SRH,RS15b}, simplifies proofs.

We begin in Section \ref{sec:1} and Section \ref{sec:2} by reviewing the concepts of Hopf algebras, Hopf pairings, and the ring of symmetric functions.  In particular we show, in Proposition~\ref{prop:extension}, that a Hopf pairing of finitely many copies of the ring of symmetric functions with itself is uniquely determined by its values on the power sum symmetric functions.  In Section \ref{sec:3}, we discuss the graded dimensions of the symmetric and exterior algebras of $\Z$-graded vector spaces, since these algebras appear in our presentations.  Finally, in Section~\ref{sec:4}, we state our main results.  Namely, we give two integral presentations of quantum lattice Heisenberg algebras in Theorems~\ref{theo:product-h-q} and~\ref{theo:product-e-h-q}.  We conclude with pointing out the connection between our results and those appearing in the literature.
\iftoggle{detailsnote}{
\medskip

\paragraph*{\textbf{Note on the arXiv version}} For the interested reader, the tex file of the arXiv version of this paper includes hidden details of some straightforward computations and arguments that are omitted in the pdf file.  These details can be displayed by switching the \texttt{details} toggle to true in the tex file and recompiling.
}{}

\section{Hopf algebras, Hopf pairings, and the Heisenberg double}
\label{sec:1}
In this section we recall the definition of the Heisenberg double.  We begin with the concepts of Hopf algebras and Hopf pairings. For the remainder of the paper, we fix a field $\kk$.  Unless otherwise indicated, vector spaces, tensor products, associative algebras, Hopf algebras, etc.\ are over this field.

\begin{defin}[Graded Hopf algebra]
  A \emph{$\Z$-graded Hopf algebra} is a tuple $(H,\nabla,\epsilon,\Delta,\eta,S)$, where $H = \bigoplus_{c \in \Z} H_c$, with each $H_c$ a finite-dimensional vector space, such that $(H,\nabla,\epsilon)$ is a $\Z$-graded algebra, $(H,\Delta,\eta)$ is a $\Z$-graded coalgebra, and $S \colon H \to H$ is an \emph{antipode}.  This means, in particular, that the following conditions are satisfied:
  \begin{gather*} \ts
    \nabla(H_c \otimes H_d) \subseteq H_{c + d},\quad
    \Delta(H_c) \subseteq \bigoplus_{d \in \Z} H_d \otimes H_{c-d},\quad \\
    \eta(\kk) \subseteq H_0,\quad \epsilon(H_c) = 0 \text{ for } c \ne 0, \\
    \nabla\circ(S\otimes1_H)\circ\Delta=\nabla\circ(1_H\otimes S)\circ\Delta=\epsilon\circ\eta.
  \end{gather*}
\end{defin}

\begin{defin}[Hopf pairing, dual Hopf pair]
  Fix two $\Z$-graded Hopf algebras $H^+$ and $H^-$. Then a bilinear map
  \[
    \langle-,-\rangle\colon H^+\times H^-\longrightarrow \kk
  \]
  is a \emph{Hopf pairing} if
  \[
    \langle-,-\rangle|_{H^+_c\times H^-_d}=0
  \]
  for all $c,d\in\Z$ and $c\neq d$, and we have
  \begin{subequations}
    \begin{align}
      \langle xy,a	\rangle&=\langle x\otimes y, \Delta(a)\rangle,\label{eq:pairingrules}\\
      \langle x,ab\rangle&=\langle \Delta(x),a\otimes b\rangle,\\
      \langle 1_{H^+},a\rangle&=\eta(a),\\
      \langle x,1_{H^-}\rangle&=\eta(x),
    \end{align}
  \end{subequations}
  for all $x,y\in H^+$ and $a,b\in H^-$.  Given such a Hopf pairing, we define
  \[
    \langle-,-\rangle\colon (H^+\otimes H^+)\otimes (H^-\otimes H^-)\longrightarrow \kk,\quad
    \langle x\otimes y,a\otimes b\rangle\mapsto \langle x,a\rangle\langle y,b\rangle.
  \]
  If the Hopf pairing is nondegenerate, we say that $(H^+,H^-)$ are a \emph{dual Hopf pair}.
\end{defin}

When dealing with Hopf algebras, it is often useful to use \emph{Sweedler notation}. For any Hopf algebra $H$, and $a \in H$, we write
\[
  \Delta(a)=\sum_{i}a_{1,i}\otimes a_{2,i}=\sum_{(a)}a_{1}\otimes a_{2} \in H \otimes H,
\]
where $a_{j,i} \in H$ for all $i$ and for $j=1,2$.

\begin{defin}[Heisenberg double]
The \emph{Heisenberg double} $\fh(H^+,H^-)$ associated to a dual Hopf pair $(H^+,H^-)$ is defined as follows. We set $\fh(H^+,H^-)=H^+\otimes H^-$ as $\kk$-modules, and write $a\#x$ for $a\otimes x$ as an element of $\fh(H^+,H^-)$, with $a\in H^+$ and $x\in H^-$. We define the multiplication of two elements of $\fh(H^+,H^-)$ by
\begin{equation}\label{eq:double-mult}
(a\# x)(b\# y)=\sum_{(x),(b)}\langle x_1, b_2 \rangle ab_1\#x_2y,
\end{equation}
extended linearly. It can be shown that this endows $\fh(H^+,H^-)$ with the structure of an associative algebra.
\end{defin}

\section{The ring of symmetric functions}
\label{sec:2}
In this section, we recall some basic facts about symmetric functions.  We refer the reader to \cite[\textsection I.1 and \textsection I.2]{Mac95} for details.

We let $\mathcal{P}$ denote the set of partitions.  We reserve Greek symbols to denote partitions. For $\lambda\in\mathcal{P}$, we let $\lambda_i$ denote the $i$-th part of the partition $\lambda$. We let $|\lambda|$ denote the sum of the parts of $\lambda$, i.e. $|\lambda|=\sum \lambda_i$. We may also write $\lambda\vdash n$ for $|\lambda|=n$. We define $\ell(\lambda)$ to be the length of $\lambda$, i.e., the number of nonzero parts of $\lambda$.  For $k \in \N$, we let $\mu\ominus k$ be the partition $\mu$ with a part $k$ removed. If $\mu$ has no part $k$, then we set $\mu\ominus k$ to be the empty partition. In an analogous manner, we define $\mu\oplus k$ to be the partition $\mu$ with a part $k$ added. We then let $\mu\oplus\lambda=(\cdots(\mu\oplus \lambda_1)\oplus\cdots)\oplus \lambda_{\ell(\lambda)})$. Similarly $\mu\ominus\lambda=(\cdots(\mu\ominus \lambda_1)\ominus\cdots)\ominus \lambda_{\ell(\lambda)})$. We let $m_i(\lambda)$ denote the number of parts of $\lambda$ equal to $i$.

Now fix a finite set $I = \{1,2,\dotsc,|I|\}$.  We let bold Greek symbols denote elements of $\mathcal{P}^{I}$. For $\blambda\in\mathcal{P}^{I}$, we let $\lambda^i$ denote the $i$-th partition, $i\in I$. We set $|\blambda|=\sum_{i,k}\lambda^i_k$ and $\ell(\blambda)=\sum_i \ell(\lambda^i)$. We interpret $\blambda\oplus\bmu$ to be the element of $\mathcal{P}^I$ whose $i$-th partition is $\lambda^i\oplus\mu^i$ for all $i\in I$. Finally, we let $\blambda\oplus (a)^j$ be the element of $\mathcal{P}^I$ whose $i$-th partition is $\lambda^i\oplus \delta_{ij}a$. Note that $\lambda\oplus 0=\lambda$.

Following~\cite[p.~1068]{SRH}, we define $\underline{\blambda}$ to be the sequence of elements of $\N_{>0} \times I$, in lexicographical order, such that the element $(k,i)$ appears $m_k(\lambda^i)$ times.  In other words $\underline{\blambda}$ consists of all the parts of the $\lambda^i$, $i \in I$, with their ``color'' $i$ recorded.

We will write the $n$-th term of the sequence $\underline{\blambda}$ as
\[
  \underline{\blambda}_n = (\prt(\underline{\blambda}_n),\clr(\underline{\blambda}_n)) \in \N_{>0} \times I,\quad \text{for } n = 1,\dotsc,\ell(\blambda).
\]

Let $\Sy$ denote the ring of symmetric functions in countably many variables $x_1, x_2, \dotsc$ with coeffcients in $\mathbb{Z}$ and define $\Sy_\Q = \Q \otimes_\Z \Sy$.  It is known that $\Sy_\Q$ is generated as a $\Q$-algebra by the \emph{power sums}:
\[
  \Sy_\Q = \Q\otimes_\Z\Sy=\Q[p_1,p_2,\dotsc],\quad p_n=\sum_{i=1}^{\infty}x_i^n,\quad n\in \N.
\]

Consider the tensor product Hopf Algebra $\Sy_\Q^{\otimes I}$.  We have that $\Sy_\Q^{\otimes I}$ is generated by $\{p_{n,i} \mid n \in \N_{> 0},\ i \in I\}$, where here, and in what follows, the second subscript indicates the factor in which an element of $\Sy$ lives.  For example, $p_{n,i}$ denotes the simple tensor in $\Sy_Q^{\otimes I}$ with $i$-th component equal to $p_n$ and all other components equal to $1$.

The coproduct on $\Sy_\Q^{\otimes I}$ is determined by
\begin{equation}\label{eq:power-cop}
  \Delta(p_{n,i})=1\otimes p_{n,i}+p_{n,i}\otimes1.
\end{equation}
For $\lambda\in\mathcal{P}$ and $\blambda \in \mathcal{P}^I$, define
\[
  p_\lambda=\prod_{i=1}^{\ell(\lambda)}p_{\lambda_i},\quad p_\blambda = \prod_{i \in I} p_{\lambda^i,i}.
\]
Thus, the $p_\blambda$, $\blambda \in \mathcal{P}^I$, form a basis for $\Sy_\Q^{\otimes I}$.

The \emph{complete symmetric functions} $h_r$, $r\in\N$, and \emph{elementary symmetric functions} $e_n$, $n\in\N$, are defined by
\begin{equation} \label{eq:sum}
  h_r=\sum_{|\lambda|=r}\frac{p_\lambda}{z_\lambda}, \qquad
  e_n=\sum_{\lambda\vdash n}(-1)^{|\lambda|-\ell(\lambda)}\frac{p_\lambda}{z_\lambda},
\end{equation}
where
\begin{equation}
  z_\lambda=\prod_{i\geq 1}i^{m_i(\lambda)} m_i(\lambda)!.\label{eq:zeta}
\end{equation}
By convention, we set $e_0 = p_0 = h_0 = 1$.  We have
\[
  \Sy = \Z[h_1,h_2,\dotsc] = \Z[e_1,e_2,\dotsc,].
\]
In other words, the $h_n$ and $e_n$ generated $\Sy$ over $\Z$, as opposed to the $p_n$, which only generated $\Sy_\Q$ over $\Q$.  We have (see, for example, \cite[Prop.~3.2]{ZM15})
\begin{equation} \label{eq:cop}
  \Delta(h_n)=\sum_{r=0}^{n}h_r\otimes h_{n-r},\qquad
  \Delta(e_n)=\sum_{r=0}^{n}e_r\otimes e_{n-r}.
\end{equation}

Let us now grab two copies of $\Sy_\Q^{\otimes I}$, rename them $H^+$ and $H^-$, and consider the associated Heisenberg double $\fh(H^+,H^-)$. For the rest of the paper, we shall consider bilinear forms on $\Sy_\Q^{\otimes I}\otimes\Sy_\Q^{\otimes I}$ of the type described in the following proposition.

\begin{prop}\label{prop:extension}
  Let $I$ be a finite set. Fix a $C_{j,i}^n\in\kk$ for every $i,j\in I$ and $n\in\N$.  Then there exists a unique Hopf pairing $\langle-,-\rangle\colon \Sy_\Q^{\otimes I}\otimes\Sy_\Q^{\otimes I}\longrightarrow \kk$ such that
  \begin{equation}\label{eq:ortogonality}
    \langle p_{n,j},p_{m,i}\rangle=\delta_{nm}C^n_{j,i}.
  \end{equation}
  This pairing is given by
  \begin{equation}\label{eq:segunda-extrapolacion}
    \langle p_\blambda,p_\bmu\rangle = \delta_{\ell(\blambda),\ell(\bmu)} \sum_{\sigma\in S_{\ell(\blambda)}} \prod_{r=1}^{\ell(\blambda)} \delta_{\prt(\underline{\blambda}_r),\prt(\underline{\bmu}_{\sigma(r)})} C^{\prt(\underline{\blambda}_r)}_{\clr(\underline{\blambda}_r),\clr(\underline{\bmu}_{\sigma(r)})},
  \end{equation}
  for all $\blambda,\bmu\in\mathcal{P}^I$. In particular
  \begin{equation}\label{eq:primera-extrapolacion}
    \langle p_{\lambda,j},p_{\mu,i}\rangle = \delta_{\lambda\mu} \prod_{k\geq1}(C^k_{j,i})^{m_k(\lambda)}m_k(\lambda)!.
  \end{equation}
\end{prop}

\begin{proof}
  The coproduct of the $p_\lambda$ is given by (cf.~\cite[Eq.~3.2]{Mac95})
  \[
    \Delta(p_\lambda)=\sum_{\mu\oplus\nu=\lambda}p_\mu\otimes p_\nu\prod_{i=1}^{\lambda_1}\binom{m_i(\lambda)}{m_i(\mu)}.
  \]
  Thus, we have
  \begin{align*}
    \langle &p_\blambda \otimes p_\bnu, \Delta(p_\bmu) \rangle \\
    &= \sum_{\balpha \oplus \bbeta = \bmu} \langle p_\blambda, p_\balpha \rangle \langle p_\bnu, p_\bbeta \rangle \prod_{j\in I} \prod_{k\geq1} \binom{m_k(\mu^j)}{m_k(\alpha^j)} \\
    &= \sum_{\balpha \oplus \bbeta = \bmu}
    \delta_{\ell(\blambda), \ell(\balpha)}
    \delta_{\ell(\bnu), \ell(\bbeta)}
    \prod_{j\in I} \prod_{k\geq1} \binom{m_k(\mu^j)}{m_k(\alpha^j)} \\
    &\qquad \qquad \sum_{\substack{ \sigma \in S_{\ell(\blambda)} \\ \tau \in S_{\ell(\bnu)}}} \prod_{r=1}^{\ell(\blambda)} \prod_{s=1}^{\ell(\bnu)} \delta_{\prt(\underline{\blambda}_r),\prt(\underline{\balpha}_{\sigma(r)})} \delta_{\prt(\underline{\bnu}_s),\prt(\underline{\bbeta}_{\tau(s)})} C^{\prt(\underline{\blambda}_r)}_{\clr(\underline{\blambda}_r),\clr(\underline{\balpha}_{\sigma(r)})} C^{\prt(\underline{\bnu}_s)}_{\clr(\underline{\bnu}_s),\clr(\underline{\bbeta}_{\tau(s)})} \\
    &= \sum_{\balpha \oplus \bbeta = \bmu}
    \delta_{\ell(\blambda), \ell(\balpha)} \delta_{\ell(\bnu), \ell(\bbeta)}
    \prod_{j\in I} \prod_{k\geq1} \binom{m_k(\mu^j)}{m_k(\alpha^j)}
    \sum_{\sigma \in S_{\ell(\blambda)} \times S_{\ell(\bnu)}}
    \prod_{r=1}^{\ell(\blambda \oplus \bnu)}
    \delta_{\prt(\underline{\blambda \oplus \bnu}_r), \prt(\underline{\bmu}_{\sigma(r)})}
    C^{\prt(\underline{\blambda \oplus \bnu}_r)}_{\clr(\underline{\blambda \oplus \bnu}_r), \clr(\underline{\bmu}_{\sigma(r)})} \\
    &= \delta_{\ell(\blambda \oplus \bnu), \ell(\bmu)}
    \sum_{\sigma \in S_{\ell(\bmu)}}
    \prod_{r=1}^{\ell(\blambda \oplus \bnu)}
    \delta_{\prt(\underline{\blambda \oplus \bnu}_r), \prt(\underline{\bmu}_{\sigma(r)})}
    C^{\prt(\underline{\blambda \oplus \bnu}_r)}_{\clr(\underline{\blambda \oplus \bnu}_r), \clr(\underline{\bmu}_{\sigma(r)})} \\
    &= \langle p_{\blambda \oplus \bmu}, p_\bnu \rangle = \langle p_\blambda p_\bmu, p_\bnu \rangle.
  \end{align*}
  Thus, \eqref{eq:segunda-extrapolacion} is a Hopf pairing.  It is clear that \eqref{eq:segunda-extrapolacion} reduces to \eqref{eq:primera-extrapolacion} and  \eqref{eq:ortogonality}.

  Now, the relation $\langle p_\blambda p_\bmu, p_\bnu \rangle = \langle p_\blambda \otimes p_\bnu, \Delta(p_\mu) \rangle$ implies that the pairing between any $p_\blambda$ and $p_\bmu$, where $\blambda$ and $\bmu$ have length greater than one, is determined by the pairing between $p_\bnu$ of strictly smaller length.  It follows by induction that the pairing is uniquely determined by \eqref{eq:ortogonality}.
\end{proof}

For $a \in \Sy_\Q$, we will write $a^\pm$ for the corresponding element of $H^\pm$. We regard $H^+$ and $H^-$ as subalgebras of $\fh(H^+,H^-)$, via the natural maps $a^+\mapsto a\#1$ and $a^-\mapsto 1\#a$ for all $a\in \Sy_\Q^{\otimes I}$.

\begin{rem}\label{rem:involution}
  The ring isomorphism $\Omega\colon \Sy_\Q\longrightarrow \Sy_\Q$, $p_n\mapsto (-1)^{n-1}p_n$ for all $n\geq1$, is an algebra involution. This follows from the fact that $\Omega$ simply rescales the variables $p_n$ of the polynomial algebra $\Sy_\Q$, and any such rescaling is an algebra isomorphism.  We denote the induced isomorphisms of $\Sy_\Q^{\otimes I}$ and $\Sy_\Q^{\otimes I} \otimes \Sy_\Q^{\otimes I}$ again by $\Omega$.
\end{rem}

\begin{lem}\label{lem:change}
  We have $\Omega(h_n)=e_n$ and $\Omega(e_n)=h_n$ for all $n \in \N$.
\end{lem}

\begin{proof}
  The proof follows directly from \eqref{eq:sum}.
  \details{
  \begin{align*}
  \Omega(h_n)&~~=\Omega\left(\sum_{\lambda\vdash n}\frac{p_\lambda}{z_\lambda}\right)=\sum_{\lambda\vdash n}\frac{1}{z_\lambda}\prod_{i=1}^{\ell(\lambda)}\Omega(p_{\lambda_i})=\sum_{\lambda\vdash n}\frac{1}{z_\lambda}\prod_{i=1}^{\ell(\lambda)}(-1)^{\lambda_i-1}p_{\lambda_i},\\
  &~~=\sum_{\lambda\vdash n}(-1)^{\sum_{i=1}^{\ell(\lambda)}(\lambda_i-1)}\frac{p_\lambda}{z_\lambda}=\sum_{\lambda\vdash n}(-1)^{|\lambda|-\ell(\lambda)}\frac{p_\lambda}{z_\lambda},\\
  &\stackrel{\eqref{eq:sum}}{=}e_n.
  \end{align*}}
\end{proof}

\begin{lem}\label{lem:preserves}
  The map $\Omega$ preserves any bilinear form $\langle - , - \rangle \colon H^+ \times H^- \longrightarrow \kk$ for which the $p_\blambda$ are orthogonal.
\end{lem}

\begin{proof}
  We have
  \begin{align*}
    \Omega(p_\blambda)&=\Omega\left(\prod_{i\in I}\prod_{k=1}^{\ell(\lambda^i)}p_{\lambda_k^i}\right)=\prod_{i\in I}\prod_{k=1}^{\ell(\lambda^i)}(-1)^{\lambda_k^i-1}p_{\lambda_k^i}=(-1)^{|\blambda|-\ell(\blambda)}p_\blambda.
  \end{align*}
  Thus, since $\langle p_\blambda,p_\bmu\rangle\in\delta_{\blambda\bmu}\Q(q)$,
  \begin{align*}
    \langle \Omega(p_\blambda),\Omega(p_\bmu)\rangle&=(-1)^{2(|\blambda|-\ell(\lambda))}\langle p_\blambda,p_\bmu\rangle=\langle p_\blambda,p_\bmu\rangle.\qedhere
  \end{align*}
\end{proof}

\begin{lem}\label{lem:compatibility}
The map $\Omega$ is compatible with the coproduct; i.e.
\[
  \Delta(\Omega(a)) = \Omega(\Delta(a)),\quad a\in\Sy_\Q^{\otimes I}.
\]
\end{lem}

\begin{proof}
  Since $\Omega$ and $\Delta$ are both algebra homomorphisms, we need only prove the claim on a generating set, and we choose the $h_n$.  Since the coproduct of a bialgebra (in this case a Hopf algebra) is an algebra homomorphism, we have
  \[
    \Omega \left(\Delta(h_n)\right)
    \stackrel{\eqref{eq:cop}}{=} \Omega\left(\sum_{r=0}^nh_r\otimes h_{n-r}\right)
    =\sum_{r=0}^ne_r\otimes e_{n-r}
    = \Delta(e_n)
    = \Delta\left(\Omega(h_n)\right).\qedhere
  \]
\end{proof}

\begin{prop}\label{prop:endomorphism}
The map $\Omega$ is an algebra endomorphism of the Heisenberg double $\fh(H^+,H^-)$, i.e. $\Omega(a\#b)\Omega(c\#d)=\Omega((a\#b)(c\#d))$ for all $a,b,c,d\in\Sy_\Q$.
\end{prop}

\begin{proof}
  Since $\{p_\blambda\}_{\blambda\in\mathcal{P}^I}$ is a basis for $\Sy_\Q$, a basis for $\fh(H^+,H^-)$ is $\{p_\blambda\#p_\bmu\}_{\blambda,\bmu\in\mathcal{P}^I}$, and we only need to prove the proposition for the elements of the latter. We have
  \begin{align*}
    \Omega((p_\blambda\#p_\bmu)( p_\bnu\#p_\bomega))&=\Omega\left(\sum_{(p_\bmu),(p_\bnu)}\langle (p_\bmu)_1,(p_\bnu)_2 \rangle p_\blambda(p_\bnu)_1\#p_\bomega (p_\bmu)_2\right) \\
    &=\sum_{(p_\bmu),(p_\bnu)}\langle \Omega((p_\bmu)_1),\Omega((p_\bnu)_2) \rangle \Omega(p_\blambda)\Omega((p_\bnu)_1)\#\Omega(p_\bomega) \Omega((p_\bmu)_2) \\
    &= \sum_{(p_\bmu),(p_\bnu)}\langle (\Omega(p_\bmu))_1,(\Omega(p_\bnu))_2 \rangle \Omega(p_\blambda)(\Omega(p_\bnu))_1\#\Omega(p_\bomega) (\Omega(p_\bmu))_2 \\
    &=\Omega(p_\blambda\#p_\bmu)\Omega( p_\bnu\#p_\bomega),
  \end{align*}
  where we have used Lemma~\ref{lem:preserves} in the second equality and Lemma~\ref{lem:compatibility} in the third equality.
\end{proof}

The following generating functions for the complete symmetric functions, the elementary symmetric functions, and the power sums can be found in~\cite[p.~19--23]{Mac95}:
\begin{subequations}
  \begin{align}
    E(t)&=\sum_{r\geq0}e_rt^r=\prod_{i\geq1}(1+x_it),\label{eq:gene}\\
    H(t)&=\sum_{r\geq0}h_rt^r=\prod_{i\geq1}(1-x_it)^{-1},\label{eq:genh}\\
    P(t)&=\sum_{r\geq1}p_rt^{r-1}=\sum_{i\geq1}\frac{\dif}{\dif t}\log\left(\frac{1}{1-x_it}\right).\label{eq:genp}
  \end{align}
\end{subequations}

\begin{lem}\label{lem:relgen}
  We have
  \begin{equation}
    H(t)=\exp\left(\sum_{r\geq1}p_r\frac{t^r}{r}\right).
  \end{equation}
\end{lem}

\begin{proof}
  We have, by \eqref{eq:genh} and \eqref{eq:genp},
  \begin{gather*}
    \int P(t) \dif t = \log\left(\prod_{i\geq1}\frac{1}{1-x_it}\right)=\log(H(t)),\\
    H(t) = \exp\left(\int P(t)\dif t\right)=\exp\left(\sum_{r\geq1}p_r\frac{t^r}{r}\right).\qedhere
  \end{gather*}
\end{proof}

\begin{lem}\label{lem:relgene}
  We have
  \begin{equation}
    E(t)=\exp\left(-\sum_{r\geq1}p_r\frac{(-t)^r}{r}\right).
  \end{equation}
\end{lem}

\begin{proof}
  The proof follows from the fact that $E(-t)H(t)=1$, and Lemma~\ref{lem:relgen}.
\end{proof}

\section{Symmetric and exterior algebras}
\label{sec:3}
To properly describe our approach to the presentation of quantum lattice Heisenberg algebras, we first introduce some ingredients used in the construction. For the remainder of the paper, we let $q$ denote an indeterminate.

\begin{defin}[Graded dimension]
Let $V=\bigoplus_{n\in\Z}V_n$ be a $\Z$-graded vector space over $\kk$. The \emph{graded dimension} of $V$ is
\[\qdim(V)=\sum_{n\in\Z}\dim(V_n)q^n\in\N[q,q^{-1}].\]
\end{defin}

\begin{defin}[Symmetric algebra]
Let $W$ be a vector space over $\kk$. The \emph{symmetric algebra} of $W$ is the quotient algebra of the tensor algebra of $W$ by the ideal $\mathcal{I}$ generated by the elements $a\otimes b-b\otimes a$, with $a,b\in W\subseteq\mathcal{T}(W)$. This is
\[S(W)=\frac{\bigoplus_{k\geq 0}\mathcal{T}^k(W)}{\langle a\otimes b-b\otimes a\rangle}.\]
The \emph{$k$-th symmetric power} of $W$, $S^k(W)$, is the image of $\mathcal{T}^k(W)$ under the quotient map $\mathcal{T}(W)\longrightarrow \mathcal{T}(W)/\mathcal{I}$.
\details{We denote the image of $a \otimes b$ in $S(W)$ by $a \circ b$. Thus, $a\circ b=b\circ a$ for all $a,b\in W$.}
\end{defin}

The generating function for the graded dimension of the symmetric powers of a $\Z$-graded vector space $V$ is
\begin{equation}\label{eq:gen-f-gr}
\sum_{k\in\N}t^k\qdim \left(S^k(V)\right)=\prod_{n\in\Z}\left(\frac{1}{1-q^nt}\right)^{\dim(V_n)}.
\end{equation}

\begin{defin}[Exterior algebra]
Let $W$ be a vector space over $\kk$. The \emph{exterior algebra} or \emph{antisymmetric algebra} of $W$ is the quotient algebra of the tensor algebra of $W$ by the ideal $\mathcal{I}$ generated by the elements $a\otimes a$, with $a\in W\subseteq\mathcal{T}(W)$. This is
\[\Lambda(W)=\frac{\bigoplus_{k\geq 0}\mathcal{T}^k(W)}{\langle a\otimes a\rangle}.\]
The \emph{$k$-th exterior power} of $W$, $\Lambda^k(W)$, is the image of $\mathcal{T}^k(W)$ under the quotient map $\mathcal{T}(W)\longrightarrow \mathcal{T}(W)/\mathcal{I}$. \details{We denote the image of $a \otimes b$ in $\Lambda(W)$ by $a \wedge b$. Thus, $a\wedge b=-b\wedge a$ for all $a,b\in W$.}
\end{defin}

The generating function for the graded dimension of the exterior powers of a $\Z$-graded vector space $V$ is
\begin{equation}\label{eq:gen-f-wedge}
\sum_{k\in\N}t^k\qdim \left(\Lambda^k(V)\right)=\prod_{n\in\Z}\left(1+q^nt\right)^{\dim(V_n)}.
\end{equation}

Since our main theorem will involve graded dimensions of symmetric and exterior powers, it is useful to have explicit expressions for these.

Consider the set of functions from $\Z$ to $\N$, $\N^\Z$. Let bold Latin characters denote functions $\bm\in\N^\Z$ with $\bm(n)=0$ for all but finitely many $n\in\Z$. We then set, for all such $\bm\in\N^\Z$, $|\bm|=\sum_{n\in\Z}\bm(n)$ and $\|\bm\|=\sum_{n\in\Z}n\bm(n)$, as well as
  \[
    s^\bm V = \prod_{n\in\Z}\binom{\dim(V_n)+\bm(n)-1}{\bm(n)},\qquad
    \ell^\bm V = \prod_{n\in\Z}\binom{\dim(V_n)}{\bm(n)},
  \]
  with $V$ a $\Z$-graded vector space of finite dimension.

\begin{prop}\label{prop:graded}
  Suppose $V$ is a finite-dimensional $\Z$-graded vector space.  Then
  \begin{equation} \label{eq:gr-sim-ext-explicit}
    \qdim(S^k(V))=\sum_{|\bm|=k}q^{\|\bm\|}s^\bm V
    \quad \text{and} \quad
    \qdim(\Lambda^k(V))=\sum_{|\bm|=k}q^{\|\bm\|}\ell^\bm V.
  \end{equation}
\end{prop}

\begin{proof}
  The Taylor series of $(1-q^n t)^{-\beta}$ about $t=0$, with $n\in\Z$, $\beta\in\N$, is
  \begin{equation}\label{eq:taylor-ab}
    (1-q^n t)^{-\beta}=\sum_{m\geq 0}t^mq^{nm}\binom{\beta+m-1}{m}.
  \end{equation}
  Substitution of \eqref{eq:taylor-ab} into \eqref{eq:gen-f-gr} yields
  \[
    \sum_{k\in\N}t^k\qdim( S^k(V))
    = \prod_{n \in \Z} \sum_{m_n \in \N} t^{m_n} q^{n m_n} \binom{\dim(V_n)+m_n-1}{m_n} \\
    = \sum_{\bm\in\N^\Z} t^{|\bm|} q^{\|\bm\|} s^\bm V.
  \]
  Comparing coefficients of powers of $t$ gives the first equation in \eqref{eq:gr-sim-ext-explicit}.  The proof of the second is analogous.
  \details{
    The Taylor series of $(1+q^n t)^{\beta}$ about $t=0$, with $n\in\Z$, $\beta\in\N$, is
    \begin{equation}\label{eq:taylor-abc}
      (1+q^n t)^{\beta}=\sum_{m\geq 0}t^nq^{nm}\binom{\beta}{m}.
    \end{equation}
    Substitution of \eqref{eq:taylor-abc} into \eqref{eq:gen-f-wedge} yields
    \[
      \sum_{k\in\N}t^k\qdim(\Lambda^k(V))
      = \prod_{n\in\Z}\sum_{m_n\in\N}t^{m_n}q^{n{m_n}}\binom{\dim(V_n)}{m_n}
      = \sum_{\bm\in\N^\Z}t^{|\bm|}q^{\|\bm\|}\ell^\bm V.
    \]
   Comparing coefficients of powers of $t$ gives the second equation in \eqref{eq:gr-sim-ext-explicit}.
  }
\end{proof}

\section{Presentations of quantum lattice Heisenberg algebras}
\label{sec:4}
In this section we present our main result on presentations of quantum lattice Heisenberg algebras.  For $n \in \N$, we let $[n]$ denote the quantum integer
\[
  [n]=\frac{q^{-n}-q^n}{q^{-1}-q}=q^{-n+1}+q^{-n+3}+\dotsb+q^{n-3}+q^{n-1}\in \N[q,q^{-1}].
\]
Note that if we set $q=1$ in the last expression above, we obtain $[n]=n$.  We define $[-n]=[n]$ for $n\in \N_{>0}$.

Let us consider the pairing
\begin{equation}\label{eq:general-pairing}
  \langle p_{n,i}, p_{m,j}\rangle=\delta_{nm}[n\langle i,j\rangle]\frac{n}{[n]},
\end{equation}
which, by Proposition~\ref{prop:extension}, extends to a unique Hopf pairing $\langle -, - \rangle \colon H^+ \times  H^- \to \Z[q,q^{-1}]$.

\begin{defin}[Quantum lattice Heisenberg algebra]\label{def:latticeHA}
  Let $L$ be the free $\Z$-module on the the set $\{v_i\}_{i\in I}$.  Let $\langle -,-\rangle_L\colon L\times L\longrightarrow \Z$ be a symmetric bilinear form. The \emph{quantum lattice Heisenberg algebra} $\fh^L_q$ associated to $L$ is the unital $\Z$-algebra generated by $q_{i,n}$, $i\in I$, $n\in \Z\setminus \{0\}$, and commutation relations
\[q_{m,i}q_{n,j}=q_{n,j}q_{m,i}+\delta_{m,-n}[n\langle v_i,v_j\rangle_L]\frac{n}{[n]},\quad i,j\in I,\quad n,m\in\Z\setminus\{0\}.\]
\end{defin}

\begin{lem}\label{lem:qha}
  With the pairing defined in \eqref{eq:general-pairing}, the Heisenberg double $\fh(H^+,H^-)$ is isomorphic to the quantum lattice Heisenberg algebra via the map
  \[
    p_{m,i}^- \mapsto q_{-m,i},\quad p_{m,i}^+ \mapsto q_{m,i},\quad m \in \N_{>0},\ i \in I.
  \]
  In particular,
  \begin{equation} \label{eq:pairing}
    p^-_{m,i}p^+_{n,j}=p^+_{n,j}p^-_{m,i}+\delta_{nm}[n\langle i,j\rangle]\frac{n}{[n]},\quad n,m \in \N_{> 0},\ i,j \in I.
  \end{equation}
\end{lem}

\begin{proof}
  We have, recalling \eqref{eq:double-mult} and \eqref{eq:power-cop},
  \begin{align*}
    p_{m,i}^-p^+_{n,j} &\ =(1\#p_{m,i})(p_{n,j}\#1),\\
    &\ = \sum_{\stackrel{(p_{m,i})}{(p_{n,j})}}\langle (p_{m,i})_1,(p_{n,j})_2\rangle (p_{n,j})_1\#(p_{m,i})_2,\\
    &\ =\langle p_{m,i},1\rangle p_{n,j}\#1+\langle 1,p_{n,j}\rangle 1\#p_{m,i}+\langle 1,1\rangle p_{n,j}\#p_{m,i}+\langle p_{m,i},p_{n,j}\rangle 1\#1,\\
    &\stackrel{\eqref{eq:general-pairing}}{=}p^+_{n,j}p^-_{m,i}+p_{n,j}\#\delta_{m0}+\delta_{0n}\#p_{m,i}+\delta_{nm}[n\langle i,j\rangle]\frac{n}{[n]},\\
    &\ =p^+_{n,j}p^-_{m,i}+\delta_{nm}[n\langle i,j\rangle]\frac{n}{[n]}.\qedhere
  \end{align*}
\end{proof}

The following two theorems give presentations of the quantum lattice Heisenberg algebra in terms of the complete and elementary symmetric functions.

\begin{theo}\label{theo:product-h-q}
  The Heisenberg double $\fh(H^+,H^-)$ is generated by the complete symmetric functions $\{h_n^+,h_n^-\}_{n\in\N}$, with relations
  \begin{align*}
    h^-_{n,j}h^+_{m,i}=\sum_{r=0}^{\min(m,n)}\qdim(S^r(V))h^+_{m-r,i}h^-_{n-r,j}\quad\textrm{if }\langle i,j\rangle\in\Z_{\geq0},\\
    h^-_{n,j}h^+_{m,i}=\sum_{r=0}^{\min(m,n)}\qdim(\Lambda^r(V))h^+_{m-r,i}h^-_{n-r,j}\quad\textrm{if }\langle i,j\rangle\in\Z_{<0},
  \end{align*}
  where $V$is a $\Z$-graded vector space such that $\qdim(V)=[\langle i,j\rangle]$.
\end{theo}

\begin{proof}
  For notational convenience, let $\chi:=\langle i,j\rangle$. First set
  \begin{align*}
    A(\zeta):=\sum_{k\geq1}p^+_{k,i}\frac{\zeta^k}{k},\qquad
    B(\omega):=\sum_{k\geq1}p^-_{k,j}\frac{\omega^k}{k}.
  \end{align*}
  Now, by Lemma \ref{lem:relgen}, we have
  \[
    \sum_{n,m\geq 0}h^-_{n,j}h^+_{m,i}\omega^n\zeta^m=\exp(B)\exp(A).
  \]
  Let us calculate
  \begin{align*}
    [B,A]&=\left(\sum_{k\geq1}p^-_{k,j} \frac{\omega^k}{k}\right) \left(\sum_{r\geq1}p^+_{r,i}\frac{\zeta^r}{r}\right) - \left(\sum_{r\geq1}p^+_{r,i}\frac{\zeta^r}{r}\right) \left(\sum_{k\geq1}p^-_{k,j}\frac{\omega^k}{k}\right) \\
    &=\sum_{k,r\geq1}(p^-_{k,j}p^+_{r,i}-p^+_{r,i}p^-_{k,j})\frac{\omega^k\zeta^r}{kr} \\
    &\stackrel{\eqref{eq:pairing}}{=} \sum_{k,r\geq1}\delta_{kr}[k\chi] \frac{k}{[k]r}\omega^k\zeta^r=\sum_{k\geq1}\frac{[k\chi]}{[k]k}(\omega\zeta)^k \\
    &=
    \begin{cases}
      \sum_{k\geq 1}\frac{(\omega\zeta)^k}{k}\frac{q^{-k\chi}-q^{k\chi}}{q^{-k}-q^k} & \textrm{if }\chi\in\Z_{\geq0}, \\
      \sum_{k\geq 1}\frac{(\omega\zeta)^k(-1)^{-k\chi+1}}{k} \frac{(-q)^{-k(-\chi)}-(-q)^{k(-\chi)}}{(-q)^{-k}-(-q)^k} & \textrm{if }\chi\in\Z_{<0}.\\
    \end{cases}
  \end{align*}
  For any two operators $A,B$, it is routine to prove (see, for example, \cite[Lem.\ 9.43]{Nak99}) that as long as $[A,B]$ commutes with both $A$ and $B$, then
  \begin{equation}\label{eq:exponent-rule}
    \exp(B)\exp(A)=\exp([B,A])\exp(A)\exp(B).
  \end{equation}
  We have
  \begin{align*}
    [B,A]&=
    \begin{cases}
      \sum_{k\geq 1}\sum_{g=0}^{\chi-1}\frac{1}{k}(q^{\chi-(2g+1)}\omega\zeta)^k & \textrm{ if }\chi\in\Z_{\geq0},\\
      -\sum_{k\geq 1}\sum_{g=0}^{(-\chi)-1}\frac{1}{k}(-q^{(-\chi)-(2g+1)}\omega\zeta)^k & \textrm{ if }\chi\in\Z_{<0},\\
    \end{cases}\\
    &=\begin{cases}
      \log\left(\prod_{g=0}^{\chi-1}(1-q^{\chi-(2g+1)}\omega\zeta)^{-1}\right) & \textrm{ if }\chi\in\Z_{\geq0},\\
      \log\left(\prod_{g=0}^{-\chi-1}(1+q^{(-\chi)-(2g+1)}\omega\zeta)\right) & \textrm{ if }\chi\in\Z_{<0}.\\
    \end{cases}
  \end{align*}
  By \eqref{eq:exponent-rule}, we have
  \begin{align*}
    \sum_{n,m\geq 0}h^-_{n,j}h^+_{m,i}\omega^n\zeta^m&=
    \begin{cases}
      \sum_{s,t\geq 0}h^+_{t,i}h^-_{s,j}\omega^s\zeta^t\prod_{g=0}^{\chi-1}(1-q^{\chi-(2g+1)}\omega\zeta)^{-1} & \textrm{ if }\chi\in\Z_{\geq0},\\
      \sum_{s,t\geq 0}h^+_{t,i}h^-_{s,j}\omega^s\zeta^t\prod_{g=0}^{-\chi-1}(1+q^{(-\chi)-(2g+1)}\omega\zeta) & \textrm{ if }\chi\in\Z_{<0}.\\
    \end{cases}
  \end{align*}
  Now, since $\qdim(V)=[\chi]=\sum_{g=0}^{\chi-1}q^{\chi-(2g+1)}$, we recall \eqref{eq:gen-f-gr} and \eqref{eq:gen-f-wedge} and thus write
  \[
    \sum_{n,m\geq 0}h^-_{n,j}h^+_{m,i}\omega^n\zeta^m=
    \begin{cases}
      \sum_{s,t\geq 0}h^+_{t,i}h^-_{s,j}\omega^s\zeta^t\sum_{k\geq 0}(\omega\zeta)^k\qdim \left(S^k(V)\right) & \textrm{ if }\chi\in\Z_{\geq0},\\
      \sum_{s,t\geq 0}h^+_{t,i}h^-_{s,j}\omega^s\zeta^t\sum_{k\geq 0}(\omega\zeta)^k\qdim \left(\Lambda^k(V)\right) & \textrm{ if }\chi\in\Z_{<0}.\\
    \end{cases}
  \]
  Upon comparison of the coefficients of $\omega\zeta$, we finally arrive at
  \[
    h^-_{n,j}h^+_{m,i}=\begin{cases}
    \sum_{r=0}^{\min(m,n)}h^+_{m-r,i}h^-_{n-r,j}\qdim(S^r(V)) & \textrm{ if }\chi\in\Z_{\geq0}, \\
    \sum_{r=0}^{\min(m,n)}h^+_{m-r,i}h^-_{n-r,j}\qdim(\Lambda^r(V)) & \textrm{ if }\chi\in\Z_{<0}.
    \end{cases}\qedhere
  \]
\end{proof}

\begin{theo}\label{theo:product-e-h-q}
  The Heisenberg double $\fh(H^+,H^-)$ is generated by the complete and elementary symmetric functions $\{h_n^+,e_n^-\}_{n\in\N}$, with relations
  \begin{align*}
    e^-_{n,j}h^+_{m,i}=\sum_{r=0}^{\min(m,n)}\qdim(\Lambda^r(V))h^+_{m-r,i}e^-_{n-r,j}\quad\textrm{if }\langle i,j\rangle\in\Z_{\geq0},\\
    e^-_{n,j}h^+_{m,i}=\sum_{r=0}^{\min(m,n)}\qdim(S^r(V))h^+_{m-r,i}e^-_{n-r,j}\quad\textrm{if }\langle i,j\rangle\in\Z_{<0},
  \end{align*}
  where $V$is a $\Z$-graded vector space such that $\qdim(V)=[\langle i,j\rangle]$.
\end{theo}

\begin{proof}
  The proof is the same as that of Theorem \ref{theo:product-h-q}, but with
  \begin{align*}
    A(\zeta)&:=\sum_{k\geq1}p^+_{k,i}\frac{\zeta^k}{k}, \qquad B(\omega):=-\sum_{k\geq1}p^-_{k,j}\frac{(-\omega)^k}{k}.\qedhere
  \end{align*}
  \details{
    By Lemma \ref{lem:relgene}, we have
    \[
      \sum_{n,m\geq 0}e^-_{n,j}h^+_{m,i}\omega^n\zeta^m=\exp(B)\exp(A).
    \]
    Let us calculate
    \begin{align*}
      [B,A]&=\left(-\sum_{k\geq1}p^-_{k,j}\frac{(-\omega)^k}{k}\right)\left(\sum_{r\geq1}p^+_{r,i}\frac{\zeta^r}{r}\right)-\left(\sum_{r\geq1}p^+_{r,i}\frac{\zeta^r}{r}\right)\left(-\sum_{k\geq1}p^-_{k,j}\frac{(-\omega)^k}{k}\right) \\
      &=-\sum_{k,r\geq1}(p^-_{k,j}p^+_{r,i}-p^+_{r,i}p^-_{k,j})\frac{(-\omega^k)\zeta^r}{kr} \\
      &\stackrel{\eqref{eq:pairing}}{=}-\sum_{k,r\geq1}\delta_{kr}[k\chi]\frac{k}{[k]r}(-\omega)^k\zeta^r=-\sum_{k\geq1}\frac{[k\chi]}{[k]k}(-\omega\zeta)^k \\
      &=
      \begin{cases}
        -\sum_{k\geq 1}\frac{(-\omega\zeta)^k}{k}\frac{q^{-k\chi}-q^{k\chi}}{q^{-k}-q^k} & \textrm{if }\chi\in\Z_{\geq0},\\
        -\sum_{k\geq 1}\frac{(-\omega\zeta)^k(-1)^{-k\chi+1}}{k} \frac{(-q)^{-k(-\chi)}-(-q)^{k(-\chi)}}{(-q)^{-k}-(-q)^k} & \textrm{if }\chi\in\Z_{<0} \\
      \end{cases}\\
      &=
      \begin{cases}
        -\sum_{k\geq 1}\sum_{g=0}^{\chi-1}\frac{1}{k}(-q^{\chi-(2g+1)}\omega\zeta)^k & \textrm{ if }\chi\in\Z_{\geq0}, \\
        \sum_{k\geq 1}\sum_{g=0}^{(-\chi)-1}\frac{1}{k}(q^{(-\chi)-(2g+1)}\omega\zeta)^k & \textrm{ if }\chi\in\Z_{<0} \\
      \end{cases}\\
      &=\begin{cases}
        \log\left(\prod_{g=0}^{\chi-1}(1+q^{\chi-(2g+1)}\omega\zeta)\right) & \textrm{ if }\chi\in\Z_{\geq0},\\
        \log\left(\prod_{g=0}^{-\chi-1}(1-q^{(-\chi)-(2g+1)}\omega\zeta)^{-1}\right) & \textrm{ if }\chi\in\Z_{<0}.\\
      \end{cases}
    \end{align*}
    By \eqref{eq:exponent-rule}, we have
    \begin{align*}
      \sum_{n,m\geq 0}e^-_{n,j}h^+_{m,i}\omega^n\zeta^m&=
      \begin{cases}
        \sum_{s,t\geq 0}h^+_{t,i}e^-_{s,j}\omega^s\zeta^t\prod_{g=0}^{\chi-1}(1+q^{\chi-(2g+1)}\omega\zeta) & \textrm{ if }\chi\in\Z_{\geq0},\\
        \sum_{s,t\geq 0}h^+_{t,i}e^-_{s,j}\omega^s\zeta^t\prod_{g=0}^{-\chi-1}(1-q^{(-\chi)-(2g+1)}\omega\zeta)^{-1} & \textrm{ if }\chi\in\Z_{<0}.\\
      \end{cases}
    \end{align*}
    Now, since $\qdim(V)=[\chi]=\sum_{g=0}^{\chi-1}q^{\chi-(2g+1)}$, we recall \eqref{eq:gen-f-gr} and \eqref{eq:gen-f-wedge} and thus write
    \[
      \sum_{n,m\geq 0}e^-_{n,j}h^+_{m,i}\omega^n\zeta^m=
      \begin{cases}
        \sum_{s,t\geq 0}h^+_{t,i}e^-_{s,j}\omega^s\zeta^t\sum_{k\geq 0}(\omega\zeta)^k\qdim \left(\Lambda^k(V)\right) & \textrm{ if }\chi\in\Z_{\geq0},\\
        \sum_{s,t\geq 0}h^+_{t,i}e^-_{s,j}\omega^s\zeta^t\sum_{k\geq 0}(\omega\zeta)^k\qdim \left(S^k(V)\right) & \textrm{ if }\chi\in\Z_{<0}.\\
      \end{cases}
    \]
    The proof follows upon comparison of coefficients.
  }
\end{proof}

\begin{rem}
  We can obtain presentations in terms of the generating sets $\{e_n^+, e_n^-\}_{n \in \N}$ and $\{e_n^+, h_n^-\}_{n \in \N}$ immediately by considering the algebra endomorphism $\Omega$ of Proposition~\ref{prop:endomorphism}.
\end{rem}

\begin{rem}
  Since the presentations of Theorems~\ref{theo:product-h-q} and~\ref{theo:product-e-h-q} are in terms of the $h_n$ and $e_n$ which generate $\Sy$ over $\Z$, they in fact yield integral forms of the quantum lattice Heisenberg algebra.  It is for this reason that these generating sets are important in categorification.  In fact, the presentations given in Theorems~\ref{theo:product-h-q} and~\ref{theo:product-e-h-q} correspond to those of \cite[Prop.~5.5]{RS15b}.  However, since the power sums do not seem to have a natural interpretation in the categorical setting of \cite{RS15b}, the connection to the presentation in terms of power sums is not given there.
\end{rem}

Integral presentations of quantum lattice Heisenberg algebras have been considered previously.  For example, in \cite[\textsection 2.2.1]{CL12}, the authors considered such presentations, but with the restriciton that $(\langle i , j\rangle)_{i,j \in I}$ was an extended Cartan matrix of type ADE.  In their paper, they define elements $a_i(n)$, $p_i^{(m)}$, and $q_j^{(m)}$ of the lattice Heisenberg algebra, related by the equations
\begin{align*}
[a_i(m),a_j(n)]&=\delta_{m,-n}[n\langle i,j\rangle]\frac{[n]}{n},\\
\sum_{n\geq0}p^{(n)}_iz^n=\exp\left(\sum_{m\geq 1}\frac{a_i(-m)}{[m]}\right),&\quad
\sum_{n\geq0}q^{(n)}_iz^n=\exp\left(\sum_{m\geq 1}\frac{a_i(m)}{[m]}\right).
\end{align*}
The connection to the notation of the current paper is given by
\[
  p_i^{(n)} = h^+_{n,i},\quad
  q_i^{(n)} = h^-_{n,i},\quad
  a_i(m)\frac{m}{[m]}=
  \begin{cases}
    p^-_{n,i} & \quad\textrm{if } m\in\Z_{\geq0},\\
    p^+_{m,i} & \quad\textrm{if } m\in\Z_{<0}.
  \end{cases}
\]
Thus, the results described in \cite[\textsection 2.2]{CL12} are specific cases of Theorem \ref{theo:product-h-q} and Theorem \ref{theo:product-e-h-q}.  However, because of the restricted setting, \cite{CL12} treats only the case where $\langle i,j \rangle$ is equal to 0, 1, or 2, instead of an arbitrary integer as in the current paper.  In addition, the authors of \cite{CL12} do not consider the algebra from the Heisenberg double point of view, and so their proofs are different.

Setting $q=1$ corresponds to the vector space $V$ appearing in Theorems~\ref{theo:product-h-q} and~\ref{theo:product-e-h-q} being concentrated in degree zero.  Under this specialization, Theorem \ref{theo:product-h-q} recovers \cite[Lem.~2.1]{Kru15}.  Note however, that the Heisenberg double approach of the current paper simplifies the proofs considerably.  In particular, the second part of the proof of \cite[Lem.~1.2]{Kru15}, where the author shows that no further relations exists between the generators, is an immediate consequence of the definition of the Heisenberg double and Lemma~\ref{lem:qha}.

The quantum lattice Heisenberg algebra and its nonquantum $q=1$ specialization are treated from the Heisenberg double point of view in \cite[\textsection 7 and \textsection 8]{SRH}, respectively.  However, as in \cite{CL12}, explicit presentations are only given when $\langle i, j \rangle$ is equal to 0, 1, or 2.

\section*{Acknowledgements}

The author would like to thank the immensely generous and patient supervision of Prof.\ Alistair Savage, facilitated by the great opportunity to conduct a research internship at the University of Ottawa underwritten by the MITACS Globalink program.

If one can be allowed some sentimentality on their first paper, the author would also like to thank his family and friends: \emph{ustedes son mi vida entera}.


\def\cprime{$'$}
\newcommand{\arxiv}[1]{\href{http://arxiv.org/abs/#1}{\tt arXiv:\nolinkurl{#1}}}
\newcommand{\doi}[1]{doi: \href{http://dx.doi.org/#1}{\tt \nolinkurl{#1}}}

\end{document}